\documentclass[reqno,12pt]{amsart}
\newcommand{\CC}{{\mathbb C}}

\def\bege{\begin{equation}} \def\ende{\end{equation}}
\def\begr{\begin{eqnarray}} \def\endr{\end{eqnarray}}
\usepackage{amsmath}
\usepackage{amssymb}
\usepackage{cite}

\usepackage{bbm}
\usepackage{amsmath}
\usepackage{txfonts}
\usepackage{stmaryrd}
\usepackage{amssymb}
\usepackage{amsfonts}
\usepackage{times}
\usepackage{mathrsfs}

\usepackage{amsthm}
\usepackage{enumerate}

\usepackage{framed}
\usepackage{lipsum}
\usepackage{color}

\def\CC{ \mathbb{C}}

\def\D{\mathbb{D}}
\def\N{\mathbb N}

\def\hD{\hat{\mathcal{D}}}
\def\dD{\mathcal{D}}
\def\vp{\varphi}
\def\om{\omega}

\def\p{{\prime}}
\def\up{\upsilon}
\def\lmd{\lambda}

\def\begr{\begin{eqnarray}} \def\endr{\end{eqnarray}}

\def\ol{\overline}
\newtheorem{Lemma}{Lemma}%[section]
\newtheorem{Theorem}[Lemma]{Theorem}

\newtheorem{Remark}[Lemma]{Remark}

\newcounter{other}            % Questions get letters
\newtheorem{otherth}[other]{Theorem}              % Other papers' theorems
\begin{document}
\title[]{A note on Stevi\'c-Sharma type operators between Bergman spaces with doubling weights }

\author{ Juntao Du,  Songxiao Li$\dagger$  and Zuoling Liu}

\address{Juntao Du\\ Department of mathematics, Guangdong University of Petrochemical Technology, Maoming, Guangdong, 525000,  P. R. China.}
 \email{jtdu007@163.com  }

\address{Songxiao Li\\ Department of mathematics, Shantou University, Shantou, Guangdong, 515063,  P. R. China.}
\email{jyulsx@163.com}
\address{Zuoling Liu \\  School of Mathematics, Jiaying University, Meizhou, Guangdong, 514015,  P. R. China }
\email{zlliustu@163.com}

 \begin{abstract}
 Using some estimates in [J. Funct. Anal. {\bf 278}(2020), Article No. 108401],
  we completely characterized the boundedness and compactness of the Stevi\'c-Sharma type operators with  different weights and different composition symbols between Bergman spaces induced by two-sides doubling weights.
 \thanks{$\dagger$ Corresponding author.}
 \thanks{The work was supported by NNSF of China (Nos. 12371131 and 12271328),  Guangdong Basic and Applied Basic Research Foundation (No. 2022A1515012117) and Projects of Talents Recruitment of GDUPT(No. 2022rcyj2008), Project of Science and Technology of Maoming (No. 2023417) and STU Scientific Research Initiation Grant (No. NTF23004).}
\end{abstract}
 \maketitle

\section{Introduction}
Let \( H(\mathbb{D}) \) be the space of all analytic functions defined in the open unit disk \( \mathbb{D} = \{ z \in \mathbb{C} : |z| < 1 \} \). For a nonnegative function \( \omega \in L^1([0, 1]) \), its extension to \( \mathbb{D} \), given by \( \omega(z) = \omega(|z|) \) for every \( z \in \mathbb{D} \), is known as a radial weight. The collection of doubling weights, represented by \( \hat{\mathcal{D}} \), includes all radial weights \( \omega \) that satisfy (see \cite{Pj2015})
 $$\hat{\omega}(r)\leq C \hat{\omega}\left(\frac{r+1}{2}\right),\,\,\,0\leq r<1$$
for a constant $C=C(\omega)\geq1$. Here  $\hat{\om}(z)=\int_{|z|}^1\om(t)dt $.
Additionally, if \( \omega \in \hat{\mathcal{D}} \) and satisfies the inequality
\[
\hat{\omega}(r) \geq C \hat{\omega}\left(1 - \frac{1 - r}{K}\right), \quad 0 \leq r < 1,
\]
for some constants \( K = K(\omega) > 1 \) and \( C = C(\omega) > 1 \), we refer to \( \omega \) as a two-sided doubling weight and denote it by \( \omega \in \mathcal{D} \). For any \( \lambda \in \mathbb{D} \), the Carleson square at \( \lambda \) is defined as
\[
S(\lambda) = \left\{ re^{i\theta} : |\lambda| \leq r < 1, \, |\text{Arg } \lambda - \theta| < \frac{1 - |\lambda|}{2} \right\}.
\]
For a radial weight \( \omega \), we define \( \omega(S(\lambda)) = \int_{S(\lambda)} \omega(z) \, dA(z) \). It is evident that \( \omega(S(\lambda)) \approx (1 - |\lambda|) \hat{\omega}(\lambda) \). For further properties of doubling weights, refer to \cite{Pj2015, PjRj2021adv, PjRjSk2021jga} and the references therein.

For \( 0 < p < \infty \) and a given \( \omega \in \hat{\mathcal{D}} \), the Bergman space \( A_\omega^p \) generated by a doubling weight \( \omega \) comprises all functions \( f \in H(\mathbb{D}) \) that satisfy
\[
\|f\|_{A_\omega^p}^p = \int_{\mathbb{D}} |f(z)|^p \omega(z) \, dA(z) < \infty,
\]
where \( dA \) represents the normalized area measure on \( \mathbb{D} \). We denote \( A_\alpha^p \) as the standard weighted Bergman space associated with the radial weight \( \omega(z) = (\alpha + 1)(1 - |z|^2)^\alpha \) for \( -1 < \alpha < \infty \). In this paper, we will assume that \( \hat{\omega}(z) > 0 \) for all \( z \in \mathbb{D} \); otherwise, \( A_\omega^p = H(\mathbb{D}) \).

Let \( S(\mathbb{D}) \) denote the set of all analytic self-maps of \( \mathbb{D} \). For \( n \in \mathbb{N} \cup \{0\} \), \( \varphi \in S(\mathbb{D}) \), and \( u \in H(\mathbb{D}) \), the generalized weighted composition operator \( uD_\varphi^{(n)} \) is defined as
\[
uD_\varphi^{(n)} f = u \left( f^{(n)} \circ \varphi \right), \quad f \in H(\mathbb{D}).
\]
The operator \( uD_\varphi^{(n)} \) was introduced by Zhu in \cite{ZXL2007}. This generalized weighted composition operator is also referred to as a weighted differentiation composition operator (see \cite{st1, st2, st3}). Specifically, when \( n = 0 \), \( uD_\varphi^{(n)} \) reduces to the weighted composition operator \( uC_\varphi \). Furthermore, when \( n = 0 \) and \( u \equiv 1 \), \( uD_\varphi^{(n)} \) becomes the composition operator \( C_\varphi \).
For more discussion on composition operators and  weighted composition operators, we refer to \cite{CM1995,EK2023, GMR2023,SU2014, Zhu1} and the references therein.  When $u\equiv 1$, $uD_\vp^{(n)}$ is the differentiation  composition operator $D_\vp^{(n)}$. When $u\equiv 1$ and $\vp(z)=z$, $uD_\vp^{(n)}$ is the $n$-th differentiation operator $D^{(n)}$.  So,  the generalized weighted composition operator   attracted a lot of attentions since it covers a lot of classical operators. See \cite{st1, st2, st3, ZXL2007, zxl2, zxl5, ZXL2019} for further information and results on generalized weighted composition operators on analytic function spaces.

In 2011, Stevi\'c, Sharma, and Bhat \cite{SsSaBa2011amc,SsSaBa2011amc-2} presented an operator defined as
$$T_{u_0,u_1,\vp}=u_0D_\vp^{(0)}+u_1 D_\vp^{(1)}$$
and explored its boundedness, compactness, and essential norm for the mapping $T_{u_0,u_1,\vp}:A_\alpha^p\to A_\alpha^p$ under the condition that
\begin{align}\label{0914-1}
u_0\in H^\infty\,\,\,\mbox{or}\,\,\,\sup_{z\in\D}\frac{|u_1(z)|}{1-|\vp(z)|^2}<\infty.
\end{align}
For further studies related to these operators, see, for instance, \cite{FG2022,WWG2020,WWG2020B,YyLy2015caot,ZfLy2018caot}.

Recently, in \cite{DjLsLz2024mmas}, the authors utilized Khinchin's inequality, Ger$\check{\mbox{s}}$gorin's theorem, and the atomic decomposition of weighted Bergman spaces to eliminate the restriction (\ref{0914-1}). They demonstrated that for $1<p,q<\infty$ and $\om,\up\in \dD$, the following holds:
$$\|T_{u_0,u_1,\vp}\|_{A_\om^p\to A_\up^q} \approx \|u_0D_\vp^{(0)}\|_{A_\om^p\to A_\up^q} + \|u_1 D_\vp^{(1)}\|_{A_\om^p\to A_\up^q}$$
and
$$\|T_{u_0,u_1,\vp}\|_{e,A_\om^p\to A_\up^q} \approx \|u_0D_\vp^{(0)}\|_{e,A_\om^p\to A_\up^q} + \|u_1 D_\vp^{(1)}\|_{e,A_\om^p\to A_\up^q}.$$
Here, $\|T_{u_0,u_1,\vp}\|_{e,A_\om^p\to A_\up^q}$ represents the essential norm of the operator $T_{u_0,u_1,\vp}: A_\om^p \to A_\up^q$, which quantifies the distance between $T_{u_0,u_1,\vp}$ and the compact operators acting from $A_\om^p$ to $A_\up^q$.
In the general scenario, the essential norm of a bounded operator \( T \) mapping from a Banach space \( X \) to another Banach space \( Y \) is defined as follows:
$$\|T\|_{e,X\to Y} = \inf\left\{\|T - K\|_{X\to Y} : K:X\to Y \text{ is compact}\right\}.$$
It is clear that \( T \) is compact if and only if \( \|T\|_{e,X\to Y} = 0 \).

In 2021, Choe et al. \cite{CCKY2020jfa} provided comprehensive characterizations regarding Carleson measures for the boundedness and compactness of differences of weighted composition operators operating on the standard weighted Bergman spaces over the unit disk. This work also led to a characterization of compactness for the differences of unweighted composition operators acting on Hardy spaces, addressing a question posed by Shapiro and Sundberg in \cite{SjSc1990pjm} three decades earlier.
Building upon their work, the results were generalized to various standard weighted Bergman spaces on the unit disk \cite{CCKY2021ieot} and the unit ball \cite{CCKP2024caot}, as well as to weighted Bergman spaces defined by two-sided doubling weights \cite{Cj2023bkms}.

Motivated  by these researches, we prove in this short note that
$$\|u_0C_\vp+u_1D_\psi^{(1)}\|_{A_\om^p\to L_\up^q}\approx \|u_0C_\vp\|_{A_\om^p\to L_\up^q}+\|u_1D_\psi^{(1)}\|_{A_\om^p\to L_\up^q}$$
and
$$\|u_0C_\vp+u_1D_\psi^{(1)}\|_{e,A_\om^p\to L_\up^q}\approx \|u_0C_\vp\|_{e,A_\om^p\to L_\up^q}+\|u_1D_\psi^{(1)}\|_{e,A_\om^p\to L_\up^q}$$
for $1<p,q<\infty$, $u_0,u_1\in H(\D)$, $\vp,\psi\in S(\D)$,  $\om\in\dD$ and  $\up$ being a positive Borel measure on $\D$.

The remainder of this paper is structured as follows: In Section 2, we present the essential preliminaries. Section 3 is devoted to the proof of the main theorem.

In this paper, the symbol \( C \) will denote constants that may change from one instance to another. 
For two positive functions \( f \) and \( g \), we use the notation \( f \lesssim g \) to signify 
that there exists a positive constant \( C \), independent of the variables, such that \( f \leq Cg \). 
Likewise, \( f \approx g \) indicates that both \( f \lesssim g \) and \( g \lesssim f \) hold true.\\
\section{preliminaries}

In this section, we state some lemmas which  will be used in the proof of main result of this paper.

Using the pull-back measure, the operators \( uC_\vp \) and \( uD_\vp^{(n)} \) in Bergman spaces associated with doubling weights were studied in \cite{DjLsSy2020ms,Lb2021bams}. To establish the necessary conditions for \( uC_\vp \) and \( uD_\vp^{(n)} \) to be bounded and compact, the authors employed a family of test functions \( \{f_{\lambda,\gamma}\} \) from \( A_\om^p \).

When $\om\in\hD$ and $0<p<\infty$, let
\begin{align*}
f_{\lambda,\gamma,\om,p}(z)=\left(\frac{1-|\lambda|^2}{1-\ol{\lambda}z}\right)^{\gamma}\frac{1}{\om(S(\lambda))^\frac{1}{p}},\, \,\lambda,z\in\D.
\end{align*}
According to \cite[Lemma 3.1]{Pj2015}, there exists $\gamma_\#=\gamma_\#(\om,p)$, whenever $\gamma>\gamma_\#$,
\begin{align}\label{0302-1}
\|f_{\lambda,\gamma,\omega,p}\|_{A_\omega^p}\approx 1,\,\,\,\,\lambda\in\D.
\end{align}
For the sake of simplicity, we will denote \( f_{\lambda,\gamma,\omega,p} \) as \( f_{\lambda,\gamma} \).
It is clear that if \( \gamma \) is held constant, \( f_{\lambda,\gamma} \) converges uniformly to 0 on any compact subset of \( \D \) as \( |\lambda| \) approaches 1.

When $u\in H(\D),\vp\in S(\D), 0<q<\infty, \up$ is a positive Borel measure on $\D$, for any measurable  set $E\subset\D$, let
$$\mu_{u,\vp,q,\up}(E)=\int_{\vp^{-1}(E)} |u(z)|^q d\up(z).$$
Then,
\begin{align}\label{0624-1}
\|uD_\vp^{(n)} \|_{A_\om^p\to L_\up^q}= \|D^{(n)}\|_{A_\om^p\to L_{\mu_{u,\vp,q,\up}}^q}.
\end{align}
To express \( \|D^{(n)}\|_{A_\omega^p \to L_{\mu_{u,\varphi,q,\up}}^q} \), as provided in \cite{PjRjSk2021jga}, we need to introduce some additional notations.

As usual, we define \(\rho(a,z) = |\varphi_a(z)| = \left|\frac{a-z}{1-\overline{a}z}\right|\) to represent the pseudohyperbolic distance between \(z\) and \(a\), and let \(E_s(a) = \{z \in \mathbb{D} : \rho(a,z) < s\}\) denote the pseudohyperbolic disc centered at \(a \in \mathbb{D}\) with radius \(s \in (0,1)\). A sequence \(\{a_i\}_{i=1}^\infty\) is referred to as a \(\delta\)-lattice if \(D = \bigcup_{i=1}^\infty E_\delta(a_i)\) and the sets \(E_{\delta/2}(a_i)\) are pairwise disjoint. Thus, Theorem 3 in \cite{PjRjSk2021jga} (see also \cite[Theorem 1.2]{Lb2021bams}) can be stated as follows.
\begin{otherth}\label{thB}
Let $0 < p,q<\infty, n\in\N\cup\{0\}, \om\in\dD, 0 < s < 1, \mu$ be a positive Borel
measure on $\D$. Then the following
statements hold.
\begin{enumerate}[(i)]
  \item When $0<p\leq q<\infty$, $D^{(n)}:A_\om^p\to L_\mu^q$ is bounded  if and only if $$\sup_{z\in\D}\frac{\mu(E_s(z))}{(1-|z|)^{nq}\om(S(z))^\frac{q}{p}}<\infty.$$
      Moreover,
           $$\|D^{(n)}\|_{A_\om^p\to L_\mu^q}^q\approx \sup_{z\in\D}\frac{\mu(E_s(z))}{(1-|z|)^{nq}\om(S(z))^\frac{q}{p}}.$$
  \item When $0<q<p<\infty$, the following statements are equivalent:
         \begin{enumerate}
         \item[(iia)]   $D^{(n)}:A_\om^p\to L_\mu^q$ is bounded;
         \item[(iib)]  $D^{(n)}:A_\om^p\to L_\mu^q$ is compact;
         \item[(iic)] the maximum function
           $$M_\om(\mu,n)(z):=\frac{\mu(E_s(z))}{(1-|z|)^{nq}\om(S(z))}, \,\,\,z\in\D$$
       belongs to $L_\om^\frac{p}{p-q}$.
         \end{enumerate}
       Moreover,
        $$\|D^{(n)}\|_{A_\om^p\to L_\mu^q}^q\approx \|M_\om(\mu,n)\|_{L_\om^\frac{p}{p-q}}.$$
\end{enumerate}
\end{otherth}

%So, a sequence $\{a_j\}_{j=1}^\infty\subset\DD$ is a $r$-lattice for some $0<r<\infty$ means
%$\beta(a_i,a_j)\geq \frac{r}{5}$ for all $i\neq j$ and $\D=\cup_{j=1}^\infty D(a_j,5r)$.
%Here and henceforth,
%$D(z,r)=\big\{\xi\in\D; \beta(z,\xi)<r\big\}.$

 When $\om\in\hD$ and $1<p<\infty$, by \cite[Theorem 7]{PjRj2021adv},  $(A_\om^p)^*\simeq A_\om^\frac{p}{p-1}$.
 By using \cite[Lemma 2.1]{CT2016jmaa}, we have the following lemma.

\begin{Lemma}\label{0406-2} Suppose $1<p,q<\infty$,  $\om\in\hD$, $\mu$ is a positive Borel measure on $\D$.  If $K:A_\om^p\to L_\mu^q$ is bounded,  $K$ is compact if and only if $\|K{f_n}\|_{L_\mu^q}\to 0$ as $n\to \infty$ whenever $\{f_n\}$ is bounded in $A_\om^p$ and uniformly converges to 0 on any compact subset of $\D$ as $n\to \infty$.
\end{Lemma}

Using this lemma, the essential norm of $D^{(n)}:A_\om^p\to L_\mu^q$ can be estimated in a standard way.
For the benefit of readers, we display the proof of it.
\begin{Theorem}\label{0624-2}
Let $1< p\leq q<\infty, n\in\N\cup\{0\}, \om\in\dD, 0 < s < 1, \mu$ be a positive Borel
measure on $\D$.    If $D^{(n)}:A_\om^p\to L_\mu^q$ is bounded,
$$\|D^{(n)}\|_{e,A_\om^p\to L_\mu^q}^q\approx \limsup_{|z|\to 1}\frac{\mu(E_s(z))}{(1-|z|)^{nq}\om(S(z))^\frac{q}{p}}.$$
\end{Theorem}
\begin{proof}
For any $r\in (0,1)$, let $\chi_r$ be the characteristic function of $r\D$. That is $\chi_r(z)=1$ when $|z|<r$ and $\chi_r(z)=0$ when $|z|\geq r$. By Lemma \ref{0406-2} and Theorem \ref{thB}, $\chi_r D^{(n)}:A_\om^p\to L_\mu^q$ is compact and
$$\|D^{(n)}\|_{e,A_\om^p\to L_\mu^q}^q\leq \|D^{(n)}-\chi_r D^{(n)}\|_{A_\om^p\to L_\mu^q}\approx \sup_{z\in\D}\frac{(\mu-\mu\chi_r)(E_s(z))}{(1-|z|)^{nq}\om(S(z))^\frac{q}{p}}.$$
Letting $r\to 1$, we have
$$\|D^{(n)}\|_{e,A_\om^p\to L_\mu^q}^q\lesssim \limsup_{|z|\to 1}\frac{\mu(E_s(z))}{(1-|z|)^{nq}\om(S(z))^\frac{q}{p}}.$$

On the other hand, for any compact operator $K:A_\om^p\to L_\mu^q$, for a given $\gamma>\gamma_\#$, when $\lmd\in\D$, we have
\begin{align*}
\|(D^{(n)}-K)f_{\lmd,\gamma}\|_{L_\mu^q}^q
&\gtrsim \int_{\D} |D^{(n)}f_{\lambda,\gamma}(z)|^qd\mu(z)-\int_\D |Kf_{\lambda,\gamma}(z)|^qd\mu(z)\\
&\geq \int_{E_s(\lambda)} |D^{(n)}f_{\lambda,\gamma}(z)|^qd\mu(z)-\int_\D |Kf_{\lambda,\gamma}(z)|^qd\mu(z).
\end{align*}
Letting $|\lambda|\to 1$, by Lemma \ref{0406-2},  we have
$$\|D^{(n)}-K\|_{A_\om^p\to L_\mu^q}^q\gtrsim \limsup_{|\lmd|\to 1}\frac{\mu(E_s(\lmd))}{(1-|\lmd|)^{qn}\om(S(\lmd))^\frac{q}{p}}.$$
Since $K$ is arbitrary, we get the desired lower estimate. The proof is complete.
\end{proof}

\begin{Remark}
By (\ref{0624-1}), Theorem \ref{thB} and Theorem \ref{0624-2}, when $1<p,q<\infty$, $n\in\N\cup\{0\}$, $\om\in\dD$ and $\up$ is a positive Borel measure on $\D$,   the norm and essential norm of $uD_\vp^{(n)}:A_\om^p\to L_\up^q$ can be estimated by using pull-back measure $\mu_{u,\vp,q,\up}$.
\end{Remark}

\section{The main result and proof}
In \cite{CCKY2020jfa}, the authors characterized the boundedness and compactness of the operator \( uC_\varphi - vC_\psi : A_\alpha^p \to A_\alpha^q \). To derive the necessary conditions, they decomposed \( \varphi^{-1}(E_s(a)) \) into six parts for any \( a \in \mathbb{D} \) in order to obtain a lower estimate for \( |u - vQ_b^\gamma| \). Here and throughout, we define
\begin{align}\label{0725-1}
Q_b = \frac{1 - \overline{b}\varphi}{1 - \overline{b}\psi}.
\end{align}
The process is quite complicated. For the convenience of the readers, all the symbols used are the same as those in \cite{CCKY2020jfa}. We will summarize some necessary estimates as follows.

%Suppose $\gamma>\gamma_\#$ and $s\in(0,1)$ are given.
For any $a\in \D\backslash\{0\}$, $\varepsilon, N>0$, $s\in(0,1)$, let  %$a_N=ae^{-N(1-|a|)\mathrm{i}},$
$$a_N:=ae^{-N(1-|a|)\mathrm{i}},\qquad\Gamma_N(a):=\big\{a\zeta:|\zeta|=1 \mbox{ and  } |\mbox{Arg} \zeta|\leq N(1-|a|)\big\},$$
$$\Omega_{\varepsilon,s}(a):=\left\{w\in\D:\frac{|1-\ol{a}w|}{(1+\varepsilon)^2}<\sup_{z\in E_s(a) }|1-\ol{a}z|\right\},\qquad
\mathrm{C}_{a}^+:=\{w\in\CC:\mathrm{Im}(\ol{a}w)>0\}.$$
Then, fix $s\in (0,1)$ and $\gamma>\max\{1,\gamma_\#\}$. Choose $N>0$ such that
$$\mbox{Arg}\left(1+\frac{8\mathrm{i}}{N(1-s)}\right)<\frac{\pi}{12\gamma}.$$
So, there exists  $\delta=\delta(N)\in(0,1)$ such that $\rho(a,a_N)<\delta$. By Lemma 4.30 in\cite{Zhu1}, for any   $z\in\D$,
\begin{align}\label{0619-1}
|1-\ol{a_N}z|\approx |1-\ol{a}z|.
\end{align}
By \cite[(4.9)]{CCKY2020jfa}, we can choose $\varepsilon=\varepsilon(s,N,\gamma)\in(0,1]$ such that
\begin{align}\label{4.9}
\sup_{z\in E_s(a) \atop{w\in \Omega_{1,s}(a)\cap \mathrm{C}_{a}^+}}
\left|\frac{1-\ol{a_N}z}{1-\ol{a_N}w}\right|^{\gamma}\leq \frac{1}{\sqrt{2}((1+\varepsilon)^2-1)}
\end{align}
as  $|a|\to 1$.
 Then, put
 \begin{align*}
 F_1(a)&:=\vp^{-1}( E_s(a) )\cap \psi^{-1}( \Omega_{\varepsilon,s}^+(a)),\,\,\,\,\\
 F_2(a)&:=\vp^{-1}( E_s(a) )\cap \psi^{-1}( \Omega_{\varepsilon,s}^-(a)),\,\,\,\,\\
 F_3(a)&:=\vp^{-1}( E_s(a) )\cap \psi^{-1}( \Omega_{\varepsilon,s}^\p(a)),
 \end{align*}
 in which,
  \begin{align*}
 \Omega_{\varepsilon,s}^+(a):= \Omega_{\varepsilon,s}(a)\cap \mathrm{C}_a^+,\quad
 \Omega_{\varepsilon,s}^-(a):= \Omega_{\varepsilon,s}(a)\backslash  \Omega_{\varepsilon,s}^+(a),\quad
  \Omega_{\varepsilon,s}^\p(a):= \D\backslash\Omega_{\varepsilon,s}(a).
 \end{align*}
 Therefore,
 $$\D= \Omega_{\varepsilon,s}^+(a)\cup  \Omega_{\varepsilon,s}^-(a)\cup   \Omega_{\varepsilon,s}^\p(a),\,\,\,\,\,
 \vp^{-1}( E_s(a) )=F_1(a)\cup F_2(a)\cup F_3(a).$$
To accomplish our goals, we can give a more assumption
\begin{align}\label{0622-3}
u=1,\quad v=2.
\end{align}
Keep this in mind, because we will use it several times.

On $F_1(a)$, noting that
\begin{align}\nonumber%\label{4.11}
u-vQ_{a_N}^\gamma=u(1-Q_{a_N}^\gamma)+(u-v)Q_{a_N}^\gamma,
\end{align}
we further decompose the region $F_1(a)$ into two parts, namely,
$$F_{1,1}(a):=\left\{z\in F_1(a):|u(z)||1-Q_{a_N}^\gamma(z)|\leq \frac{1}{1+\varepsilon}|u(z)-v(z)||Q_{a_N}^\gamma(z)|\right\}$$
and  $$F_{1,2}(a):=F_1(a)\backslash F_{1,1}(a).$$
By Lemma 3.3 in \cite{CCKY2020jfa}, we have
\begin{align*}
|Q_{a_N}|\geq\frac{1-|a|}{|1-\ol{a_N}\psi|}\gtrsim  1 \mbox{ on }\psi^{-1}(\Omega_{1,\varepsilon}(a)).
\end{align*}
Since $F_1(a))\subset \psi^{-1}(\Omega_{1,\varepsilon}(a))$,
\begin{align}\label{0622-4}
|Q_{a_N}|\gtrsim  1 \mbox{ on } F_1(a)).
\end{align}
Then, (4.12) and (4.13) in \cite{CCKY2020jfa} demonstrate that
\begin{align}\label{0622-1}
\left|u-vQ_{a_N}^\gamma\right|^q
\approx |u|^q\left|1-Q_{a_N}^\gamma\right|^q+|u-v|^q|Q_{a_N}^\gamma|^{q\gamma}
\gtrsim 1
\,\,\mbox{ on }\,\,F_{1,1}(a)
\end{align}
and
\begin{align}\label{0622-2}
\left|u-vQ_{a_N}^{2\gamma}\right|
\gtrsim |u|^q\left|1-Q_{a_N}^\gamma\right|^q+|u-v|^q|Q_{a_N}^\gamma|^{q\gamma}
\gtrsim 1
\,\,\mbox{ on }\,\,F_{1,2}(a).
\end{align}
%Here, the last $\gtrsim$'s in (\ref{0622-1}) and (\ref{0622-2}) come from (\ref{0622-3}) and (\ref{0622-4}).
Here, the last \(\gtrsim\) in equations (\ref{0622-1}) and (\ref{0622-2}) are derived from equations (\ref{0622-3}) and (\ref{0622-4}).

Then, on $F_3(a)$, we decompose the region $F_3(a)$ into two parts as in the case of $F_1(a)$. Set
$$
F_{3,1}(a):=\left\{z\in F_3(a):|u(z)|\leq \frac{1}{1+\varepsilon}|v(z)||Q_a^\gamma(z)|\right\}
$$
and
$$
F_{3,2}(a):=F_3(a)\backslash F_{3,1}(a).
$$
Then, using the assumption (\ref{0622-3}), and applying Lemma 4.3 and equation (4.8) in \cite{CCKY2020jfa}, we obtain
\begin{align}\label{0622-5}
|u - vQ_a^\gamma|^q \approx |u|^q + |vQ_a^\gamma|^q \gtrsim 1 \quad \text{on } F_{3,1}(a);
\end{align}
and from equations (4.19), (4.20), and Lemma 4.3 in \cite{CCKY2020jfa}, we find
\begin{align}\label{0622-6}
|u - vQ_a^{\gamma+1}|^q \approx |u|^p = 1 \quad \text{on } F_{3,2}(a).
\end{align}
From the definition of the set \(\Omega_{\varepsilon,s}\), we have
\begin{align}\label{0622-7}
|Q_a| \leq \frac{1}{(1 + \varepsilon)^2} \quad \text{on } F_3(a).
\end{align}
With these estimates established, we can achieve our goal.

\begin{Theorem}\label{0421-1}
Suppose $1<p, q<\infty,\om\in\dD$, $\up$ is a positive Borel measure on $\D$, $\vp,\psi\in S(\D)$ and  $u_0,u_1\in H(\D)$. Then,
$$\|u_0C_{\vp}+u_1D^{(1)}_{\psi}\|_{A_\om^p\to L_\up^q}\approx \|u_0C_{\vp}\|_{A_\om^p\to L_\up^q}+\|u_1D^{(1)}_{\psi}\|_{A_\om^p\to L_\up^q}.$$
Moreover, if $1<p\leq q<\infty$ and $u_0C_{\vp}+u_1D^{(1)}_{\psi}:A_\om^p\to L_\up^q$ is bounded,
$$\|u_0C_{\vp}+u_1D^{(1)}_{\psi}\|_{e,A_\om^p\to L_\up^q}\approx \|u_0C_{\vp}\|_{e,A_\om^p\to L_\up^q}+\|u_1D^{(1)}_{\psi}\|_{e,A_\om^p\to L_\up^q}.$$
\end{Theorem}
\begin{proof}
The upper estimates of the norm and essential norm of $u_0C_\vp+u_1 D^{(1)}_\psi$ are obvious.
Put
$$T=u_0C_{\vp}+u_1D^{(1)}_{\psi},\,\quad \|\cdot\|_{A_\om^p\to L_\up^q}=\|\cdot\|,\quad \|\cdot\|_{e,A_\om^p\to L_\up^q}=\|\cdot\|_e $$
for simplicity.
Then we will get the  lower estimates in three parts.
%
%For any  $\lmd\in\D$, recall that, whenever $\gamma>\gamma_\#$,
%$$f_{\lmd,\gamma}(z)=\left(\frac{1-|\lmd|^2}{1-\ol{\lmd}z}\right)^\gamma\frac{1}{\om(S(\lmd))^\frac{1}{p}}.$$
%and $\|f_{\lmd,\gamma}\|_{A_\om^p}\approx 1$.

{\bf Part (a): $1<p\leq q<\infty$ and the lower estimate of $\|T\|$.}

For any given $a\in\D$, as $|a|\to 1$,  we have
\begin{align}
\|Tf_{\lmd,\gamma}\|_{L_\up^q}^q
&=\frac{1}{\om(S(\lmd))^\frac{q}{p}}\int_\D \left|\frac{u_0(1-|\lmd|^2)^\gamma}{(1-\ol{\lmd}\vp)^\gamma}+
\frac{\gamma\ol{\lmd}u_1(1-|\lmd|^2)^\gamma}{(1-\ol{\lmd}\psi)^{\gamma+1}}\right|^q d\up \nonumber\\
&\geq \frac{(1-|\lambda|^2)^{q\gamma}}{\om(S(\lmd))^\frac{q}{p}}\int_{\vp^{-1}( E_s(a) )} \left|u_0+
\frac{\gamma\ol{\lmd}u_1 (1-\ol{\lmd}\vp)^\gamma}{(1-\ol{\lmd}\psi)^{\gamma+1}}\right|^q\frac{1}{|1-\ol{\lmd}\vp|^{q\gamma}} d\up. \nonumber%\label{0619-3}
\end{align}

On $F_1$, letting $\lambda=a_N$, by (\ref{0619-1}), we get
\begin{align*}
\|Tf_{a_N,\gamma}\|_{L_\up^q}^q
&\gtrsim \frac{1}{\om(S(a_N))^\frac{q}{p}}\int_{F_1(a)} \left|u_0+
\frac{\gamma\ol{a_N}u_1 (1-\ol{a_N}\vp)^\gamma}{(1-\ol{a_N}\psi)^{\gamma+1}}\right|^q d\up
\end{align*}
and
\begin{align*}
\|Tf_{a_N,2\gamma}\|_{L_\up^q}^q
&\gtrsim \frac{1}{\om(S(a_N))^\frac{q}{p}}\int_{F_1(a)} \left|u_0+
\frac{2\gamma\ol{a_N}u_1 (1-\ol{a_N}\vp)^{2\gamma}}{(1-\ol{a_N}\psi)^{2\gamma+1}}\right|^q d\up.
\end{align*}
By (\ref{4.9}), we have
\begin{align*}
\|Tf_{a_N,\gamma}\|_{L_\up^q}^q
&\gtrsim \frac{1}{\om(S(a_N))^\frac{q}{p}}\int_{F_1(a)} \left|\frac{2u_0(1-\ol{a_N}\vp)^\gamma}{(1-\ol{a_N}\psi)^\gamma}+
\frac{2\gamma\ol{a_N}u_1 (1-\ol{a_N}\vp)^{2\gamma}}{(1-\ol{a_N}\psi)^{2\gamma+1}}\right|^q d\up.
\end{align*}
Then, the Triangle inequality implies
\begin{align}
\|Tf_{a_N,\gamma}\|_{L_\up^q}^q+\|Tf_{a_N,2\gamma}\|_{L_\up^q}^q
&\gtrsim \frac{1}{\om(S(a_N))^\frac{q}{p}}\int_{F_1(a)}\left|
u-vQ_{a_N}^\gamma
\right|^q
 |u_0|^q d\up.
\label{0619-2}
\end{align}
Here, $Q_{a_N}$ and $u,v$ are those defined in (\ref{0725-1}) and (\ref{0622-3}).
By (\ref{0622-1}) and (\ref{0622-2}), we obtain
\begin{align*}
\|Tf_{a_N,\gamma}\|_{L_\up^q}^q+\|Tf_{a_N,2\gamma}\|_{L_\up^q}^q
&\gtrsim \frac{1}{\om(S(a_N))^\frac{q}{p}}
\int_{F_{1,1}(a)}  |u_0|^q d\up
\end{align*}
and
\begin{align*}
\|Tf_{a_N,2\gamma}\|_{L_\up^q}^q+\|Tf_{a_N,4\gamma}\|_{L_\up^q}^q
&\gtrsim \frac{1}{\om(S(a_N))^\frac{q}{p}}
\int_{F_{1,2}(a)}  |u_0|^q d\up.
\end{align*}
So, the estimate (\ref{0302-1})   deduces
\begin{align*}
\|T\|^q &\gtrsim \frac{1}{\om(S(a))^\frac{q}{p}}
\int_{F_1(a)} |u_0|^q d\up.
\end{align*}

Next, on $F_2(a)$, letting $\lambda=\ol{a_N}$, by a symmetric argument, we have
\begin{align*}
\|T\|^q
&\gtrsim \frac{1}{\om(S(a))^\frac{q}{p}}
\int_{F_2(a)}
 |u_0|^q d\up.
\end{align*}

Finally, we consider the integral over $F_3(a)$.
Letting $\lambda=a$, as we get (\ref{0619-2}), replacing (\ref{4.9}) with (\ref{0622-7}), we have
\begin{align*}
\|Tf_{a,\gamma}\|_{L_\up^q}^q+\|Tf_{a,2\gamma}\|_{L_\up^q}^q
&\gtrsim \frac{1}{\om(S(a))^\frac{q}{p}}\int_{F_3(a)}\left|
u-vQ_a^\gamma
\right|^q  |u_0|^q d\up.
\end{align*}
By (\ref{0622-5})  and (\ref{0622-6}), we obtation
\begin{align*}
\|Tf_{a,\gamma}\|_{L_\up^q}^q+\|Tf_{a,2\gamma}\|_{L_\up^q}^q
&\gtrsim \frac{1}{\om(S(a))^\frac{q}{p}}
\int_{F_{3,1}(a)}  |u_0|^q d\up
\end{align*}
and
\begin{align*}
\|Tf_{a,\gamma+1}\|_{L_\up^q}^q+\|Tf_{a,2(\gamma+1)}\|_{L_\up^q}^q
&\gtrsim \frac{1}{\om(S(a))^\frac{q}{p}}
\int_{F_{3,2}(a)}
 |u_0|^q d\up.
\end{align*}
Using (\ref{0302-1}) again, we have
\begin{align*}
\|T\|^q
&\gtrsim \frac{1}{\om(S(a))^\frac{q}{p}}
\int_{F_3(a)}
 |u_0|^q d\up.
\end{align*}

Therefore, as $|a|\to 1$, we obtain
\begin{align}\label{0625-1}
\|T\|^q\gtrsim \frac{1}{\om(S(a))^\frac{q}{p}}\int_{\vp^{-1}( E_s(a) )} |u_0|^q d\up.
\end{align}
So, we can choose a constant $\tau\in(0,1)$ such that (\ref{0625-1}) holds for all $\tau<|a|<1$.
When $|a|\leq \tau$, letting $h\equiv 1$,
\begin{align}\label{0624-3}
\|T\|^q
\gtrsim \|Th\|_{L_\up^q}^q
\gtrsim \frac{1}{\om(S(a))^\frac{q}{p}}\int_{\vp^{-1}( E_s(a) )} |u_0|^q d\up.
\end{align}
Therefore, by (\ref{0624-1}) and Theorem \ref{thB}, $$\|u_0C_\vp\|\lesssim \|T\|.$$
Thus, $\|u_1D_\psi^{(1)}\|\lesssim \|T\|.$

{\bf Part (b): $1<p\leq q<\infty$ and the lower estimate of $\|T\|_e$.}

In the previous process, when $|a|>\tau$, on every region $F_{i,j}(a)(i=1,2,3;j=1,2)$, we  choose $\lambda=\lambda_{a,i}$ and $\gamma=\gamma_{i,j}>\gamma_\#$ properly to obtain
\begin{align*}
\sum_{k=1}^2\|Tf_{\lambda_{a,i},k\gamma_{i,j}}\|_{L_\up^q}^q
=\|Tf_{\lambda_{a,i},\gamma_{i,j}}\|_{L_\up^q}^q+\|Tf_{\lambda_{a,i},2\gamma_{i,j}}\|_{L_\up^q}^q
&\gtrsim \frac{1}{\om(S(a))^\frac{q}{p}}
\int_{F_{i,j}(a)} |u_0|^q d\up.
\end{align*}
When $K:A_\om^p\to L_\up^q$ is compact, letting $I=\{1,2,3\}\times\{1,2\}\times\{1,2\}$, by (\ref{0302-1}), we have
\begin{align*}
\|T-K\|^q
\gtrsim&  \sum_{(i,j,k)\in I} \|(T-K)f_{\lambda_{a,i},k\gamma_{i,j}}\|_{L_\up^q}^q \\
\gtrsim& \sum_{(i,j,k)\in I}
\left(\|Tf_{\lambda_{a,i},k\gamma_{i,j}}\|_{L_\up^q}^q - \|Kf_{\lambda_{a,i},k\gamma_{i,j}}\|_{L_\up^q}^q\right)\\
\gtrsim& \frac{1}{\om(S(a))^\frac{q}{p}}\int_{\vp^{-1}( E_s(a) )}  |u_0|^q d\up
-\sum_{(i,j,k)\in I} \|Kf_{\lambda_{a,i},k\gamma_{i,j}}\|_{L_\up^q}^q.
\end{align*}
Since $K$ is arbitrary and compact, letting $|a|\to 1$, by Lemma \ref{0406-2},  we obtain
$$\|T\|_e^q \gtrsim
\limsup_{|a|\to 1}\frac{1}{\om(S(a))^\frac{q}{p}}\int_{\vp^{-1}( E_s(a) )} |u_0|^q d\up .
$$
Then,  (\ref{0624-1}) and Theorem \ref{0624-2} deduce
$$\|T\|_e\gtrsim \|u_0C_\vp\|_e .$$
Therefore, $$\|T\|_e\gtrsim \|u_1D^{(1)}_\psi\|_e .$$

{ \bf Part (c): $1<q<p<\infty$ and the lower estimates of $\|T\|$.}

 Let $\{r_k(t)\}$  be  Rademacher functions, $\vec{a}=\{a_k\}_{k=1}^\infty$ be a $\frac{s}{2}$-lattice in $\D$
 and  $\vec{\lambda}=\{\lambda_k\}$ be a sequence given by one of $\{a_k\}$, $\{a_{k,N}\}$ and $\{\ol{a_{k,N}}\}$.
 Here, $a_{k,N}=(a_k)_N$.
 We claim that $\{a_{k,N}\}_{k=1}^\infty$ is separated. Otherwise, for any $0<x<\frac{1}{3}$, there exist $a_{i,N}$ and $a_{j,N}$ such that
 $E_x(a_{i,N})\cap E_x(a_{j,N})\neq \O$. By \cite[Lemma 4.3]{CCKY2021ieot}, $\rho(a_i,a_j)<2(1+8N)x$.
 This is contradictory to $\vec{a}$ is a lattice.
Moreover, we can assume $\inf|a_k|>0$ and $\{|a_k|\}$ is a increasing sequence.
Therefore, there exists a $M>0$ such that $|a_k|>\tau$ if and only if $k>M$.
Here, the constant $\tau$ is that decided by (\ref{0625-1}).

 Let $y$ be one of $\gamma,\gamma+1,2\gamma$. By \cite[Thoerem 1]{PjRjSk2021jga},  for any $\vec{c}=\{c_k\}_{k=1}^\infty\in l^p$,
$\|g_{\vec{c},y,t,\vec{\lambda}}\|_{A_\om^p}\lesssim \|\vec{c}\|_{l^p}$,  where
$$g_{\vec{c},y,t,\vec{\lambda}}(z)=\sum_{k=1}^\infty c_k r_k(t)f_{\lambda_k,y}(z).  $$
Let $\chi_k$ be the  characteristic function of $E_s(a_k)$.
By Fubini's theorem,  Khinchin's inequality  and (\ref{0619-1}),  we have
 \begin{align}
\int_0^1 \|T  g_{\vec{c},y,t,\vec{\lambda}}\|_{L_\up^q}^qdt
=&\int_\D\int_0^1 \left|\sum_{k=1}^\infty c_k r_k(t) T f_{\lambda_k,y}  \right|^q   dt   d\up \nonumber\\
\approx& \int_\D\left( \sum_{k=1}^\infty |c_k|^2|T  f_{\lambda_k,y}|^2\right)^\frac{q}{2}  d\up
\geq \int_\D\left( \sum_{k=1}^\infty |c_k|^2|T  f_{\lambda_k,y}|^2(\chi_k\circ\vp)\right)^\frac{q}{2}  d\up \nonumber
\\
\approx & \sum_{k=1}^\infty |c_k|^q \int_{\vp^{-1}(E_s(a_k))} |T  f_{\lambda_k,y}|^qd\up \nonumber\\
\approx&  \sum_{k=1}^\infty
\frac{|c_k|^q }{\om(S(a_k))^\frac{q}{p}}\int_{\vp^{-1}(E_s(a_k))}\left|u_0+
\frac{y\ol{\lambda_k}u_1 (1-\ol{\lambda_k}\vp)^y}{(1-\ol{\lambda_k}\psi)^{y+1}}\right|^q
\frac{(1-|\lambda_k|^2)^{qy}}{|1-\ol{\lambda_k}\vp|^{qy}} d\up.  \nonumber%\label{0625-2}
\end{align}
Then, as we did in  Part (a), when $|a_k|>\tau$, on every region $F_{i,j}(a_k)(i=1,2,3;j=1,2)$,
we can  choose $\lambda_k=\lambda_{a_k,i}$ and $y=\gamma_{i,j}>\gamma_\#$ properly to obtain
 \begin{align*}
\int_0^1 \|T  g_{\vec{c},y,t,\vec{\lambda}}\|_{L_\up^q}^qdt +  \int_0^1 \|T  g_{\vec{c},2y,t,\vec{\lambda}}\|_{L_\up^q}^qdt
\gtrsim&  \sum_{k=M}^\infty
\frac{|c_k|^q}{\om(S(a_k))^\frac{q}{p}}\int_{F_{i,j}(a_k)}\left|u_0\right|^q  d\up.
\end{align*}
Therefore,
\begin{align*}
\|T\|^q\|\vec{c}\|_{l^p}^q\gtrsim
\sum_{i=M}^\infty \frac{ |c_i|^q}{\om(S(a_i))^\frac{q}{p}}
\int_{\vp^{-1}(E_s(a_i))}  |u_0|^qd\up.
\end{align*}
By (\ref{0624-3}), we have
\begin{align*}
\|T\|^q\|\vec{c}\|_{l^p}^q
\gtrsim&  \sum_{i=1}^\infty
\frac{|c_i|^q}{\om(S(a_i))^\frac{q}{p}}\int_{\vp^{-1}(E_s(a_i))}|u_0|^q  d\up.
\end{align*}
Let $\eta_i=\frac{1}{\om(S(a_i))^\frac{q}{p}}\int_{\vp^{-1}(E_s(a_i))}|u_0|^q  d\up$.
Since $\{|c_i|^q\}\in l^\frac{p}{q}$ and $\|\{|c_i|^q\}\|_{l^\frac{p}{q}}=\|\vec{c}\|_{l^p}^q$,
we have $\{\eta_i\}\in (l^\frac{p}{q})^*=l^\frac{p}{p-q}$ and $\|\{\eta_i\}\|_{l^\frac{p}{p-q}}\leq \|T\|^q$.
Let the Borel measure $\mu$ on $\D$ be defined by
$$\mu(E)=\int_{\vp^{-1}(E)}|u_0|^qd\up.$$
We have
\begin{align*}
\int_\D \left(\frac{\mu(E_\frac{s}{2}(z))}{\om(S(z))}\right)^\frac{p}{p-q}\om(z)dA(z)
&\leq \sum_{i=1}^\infty \int_{E_\frac{s}{2}(a_i)} \left(\frac{\mu(E_\frac{s}{2}(z))}{\om(S(z))}\right)^\frac{p}{p-q}\om(z)dA(z)\\
&\lesssim \sum_{i=1}^\infty \left(\frac{\mu(E_s(a_i))}{\om(S(a_i))}\right)^\frac{p}{p-q} \om(E_{\frac{s}{2}}(a_i))\\
&\lesssim \sum_{i=1}^\infty \left(\frac{\mu(E_s(a_i))}{\om(S(a_i))^\frac{q}{p}}\right)^\frac{p}{p-q} \\
&\lesssim \|T\|^\frac{pq}{p-q}.
\end{align*}
By (\ref{0624-1}) and Theorem \ref{thB}, we have
$$\|u_0C_\vp\|\approx \|I_d\|_{A_\om^p\to L_\mu^q}
\approx \left(\int_\D \left(\frac{\mu(E_\frac{s}{2}(z))}{\om(S(z))}\right)^\frac{p}{p-q}\om(z)dA(z)\right)^\frac{p-q}{pq}\lesssim \|T\|.$$
Therefore, we also have
$$\|u_1D_\psi^{(1)}\|\lesssim \|T\|.$$
Moreover, by Theorem \ref{thB},   the boundedness of $T:A_\om^p\to L_\up^q(q<p)$ implies the compactness of it.
The proof is complete.
\end{proof}

\end{document}